\documentclass[12pt,a4paper]{amsart}
\usepackage{mathrsfs}
\usepackage{amsfonts}

\usepackage[active]{srcltx}
\usepackage{enumerate}

\usepackage{amsmath,amssymb,xspace,amsthm}

\usepackage[all]{xy}

 \usepackage[colorlinks,final,backref=page,hyperindex]{hyperref}

\newtheorem{theorem}{Theorem}[section]
\newtheorem{example}{Example}[section]
\newtheorem{remark}[theorem]{Remark}
\newtheorem{lemma}[theorem]{Lemma}
\newtheorem{proposition}[theorem]{Proposition}
\newtheorem{corollary}[theorem]{Corollary}

\newtheorem{definition}[theorem]{Definition}
\numberwithin{equation}{section} \theoremstyle{definition}
\newcommand\darrowa{\longrightarrow {\mkern -27mu} {\raise 6pt \hbox{ $\pi_0$}} {\mkern 16mu}}

\newcommand{\C}{\ensuremath{\mathbb C}\xspace}

\newcommand{\Z}{\ensuremath{\mathbb{Z}}\xspace}
\newcommand{\N}{\ensuremath{\mathbb{N}}\xspace}

\def\mn{\mathfrak{n}}

\def\mg{\mathfrak{g}}
\def\mh{\mathfrak{h}}

\def\sl{\mathfrak{sl}}
\def\gl{\mathfrak{gl}}

\def\span{\text{span}}

\def\mm{\mathfrak{m}}

\newcommand{\p}{\partial}

\def\p{\partial}

\def\bm{\mathbf{m}}
\def\bs{\mathbf{s}}
\def\ba{\mathbf{a}}
\def\br{\mathbf{r}}
\def\be{\mathbf{e}}
\def\bo{\mathbf{0}}
\def\cd{\mathcal{D}}
\def\bal{\mathbf{\alpha}}
\def\bbe{\mathbf{\beta}}
\begin{document}

\title[Whittaker modules]{  Whittaker category  for the Lie algebra of polynomial vector fields}
\author{ Yufang Zhao and Genqiang Liu }
\date{}\maketitle

\begin{abstract}For any positive integer $n$, let $A_n=\C[t_1,\dots,t_n]$,  $W_n=\text{Der}(A_n)$ and $\Delta_n=\text{Span}\{\frac{\partial}{\partial{t_1}},\dots,\frac{\partial}{\partial{t_n}}\}$.
Then $(W_n, \Delta_n)$ is a Whittaker pair.
 A $W_n$-module $M$ on which $\Delta_n$ operates locally finite  is called a Whittaker module.  We show that each block $\Omega_{\mathbf{a}}^{\widetilde{W}}$ of the category of  $(A_n,W_n)$-Whittaker modules with finite dimensional Whittaker vector spaces is equivalent to the category of finite dimensional modules over $L_n$, where $L_n$ is the Lie subalgebra of $W_n$
 consisting of vector fields vanishing at the origin. As a corollary,  we classify all simple non-singular
 Whittaker $W_n$-modules with finite dimensional Whittaker vector spaces using $\mathfrak{gl}_n$-modules.  We also obtain an analogue of Skryabin's equivalence for  the non-singular block $\Omega_{\mathbf{a}}^W$.
 \end{abstract}

\vskip 5pt
\noindent
{\em 2020  Math. Subj. Class.:}
17B10, 17B20, 17B65, 17B66, 17B68

\vskip 10pt

\section{Introduction}

Whittaker modules are important objects in the representation theory since
its emergence in the study of complex  semisimple  Lie algebras. Let $\mg=\mn_-\oplus \mh \oplus \mn_+$ be a complex semisimple  Lie algebra and let $\eta:\mn_+\rightarrow \C$ be a Lie algebra homomorphism. A $\mg$-module is called a Whittaker module if $x-\eta(x)$ acts locally nilpotent on it for any $x\in \mn_+$.  Whittaker modules were first defined for $\sl_2$ by Arnal and Pinzcon \cite{AP}.  Kostant \cite{K}  defined
Whittaker modules for all  complex semisimple Lie algebras. When $\eta$ is non-singular, Kostant showed that simple  Whittaker modules are in bijective
correspondence with maximal ideals of the center $Z(\mg)$ of the universal enveloping algebra of $\mg$. For quantum groups, Whittaker modules have been studied  in \cite{O1,S,XGZ}. Whittaker modules have also been studied for
the Virasoro algebra in \cite{OW1,OW2,GLZ,LPX,LWZ}, for Heisenberg
algebras in \cite{Ch}, for Weyl algebras and generalized Weyl algebras in \cite{BO}, for  affine Lie algebras  in \cite{ALZ, CJ} and for classical Lie-super algebras  in \cite{C}.  In \cite{BM}, the authors  gave a more general picture of Whittaker modules for many Lie algebras which includes, in particular, Lie algebras with triangular decomposition and simple Lie algebras of Cartan type. They described some basic properties of Whittaker modules, including a block decomposition of the category of Whittaker modules and some properties of simple Whittaker modules under some assumptions.

Among the theory of infinite dimensional Lie algebras,
  the Lie algebras   of  polynomial vector fields on irreducible affine algebraic varieties  is an important class of Lie
algebras. These algebras contain the  Cartan type Lie algebra $W_n$ which is the derivation Lie algebra of $A_n=\C[t_1,\dots, t_n]$. Weight modules with finite dimensional weight spaces (also called Harish-Chandra modules) for $W_n$ has been well developed.  In 1974-1975,  A. Rudakov
 addressed the classification of a class of irreducible $W_n$-representations
that satisfy some natural topological conditions, see \cite{R1,R2}.
Tensor modules $T(P, V) $ over $W_n$ were introduced by Shen and Larsson, see \cite{Sh,La}, for a $\cd_n$-module $P$ and $\gl_n$-module V, where $\cd_n$ is the Weyl algebra. In \cite{LLZ}, the structure of any $T(P, V ) $ was completely determined. When $V$ is finite dimensional, the simplicity   of $T(P, V ) $ over the Witt algebra
on an $n$-dimensional torus  was given by Rao, see \cite{E1}.
The
classification of simple Harish-Chandra modules over $W_1$ was given in \cite{M1}. Very recently,
based on the classification of simple uniformly bounded $W_n$ modules (see \cite{CG} for $n =2$  and \cite{XL1}
for any $n$) and the description  of the weight set of simple weight $W_n$ modules
(see \cite{PS}), D. Grantcharov and V. Serganova completed the classification of simple Harish-Chandra $W_n$-modules, see \cite{GS}. They proved that every such nontrivial module $M$ is a simple sub-quotient of some tensor module $T(P, V) $.
In \cite{DSY}, tilting modules and their character formulas for category $\mathcal{O}$ of $W_n$ were  described.
However non-weight modules and weight modules that are not
Harish-Chandra modules for $W_n$ are not well developed.

 In this paper, we study the category of  Whittaker  modules for $W_n$.
 Let
 $\Delta_n=\text{Span}\{\frac{\partial}{\partial{t_1}},\dots,\frac{\partial}{\partial{t_n}}\}$ which is a commutative subalgebra of $W_n$.
Then $(W_n, \Delta_n)$ is a Whittaker pair in the sense of \cite{BM}.
A $W_n$-module $M$ on which $\Delta_n$ operates locally finite  is called a Whittaker module. We denote the semidirect Lie algebra $W_n\ltimes A_n$ by $\widetilde{W}_n$.  A $\widetilde{W}_n$-module $M$ is called an $(A_n, W_n)$-module, if the action of  $A_n$  on $M$ is associative.
 This  kind of modules play an important role in the classification of
simple modules for Cartan type Lie algebras, see \cite{B,BF1,GLLZ,E2, XL1}. For an $\ba\in \C^n$ and a Whittaker $W_n$-module $M$, a nonzero  $v\in M$ is called a Whittaker vector if $\frac{\partial}{\partial{t_i}}v=a_i v $ for any $i\in\{1,\cdots,n\}$.
Let $\Omega_{\mathbf{a}}^W$ (resp. $\Omega_{\mathbf{a}}^{\widetilde{W}}$) be the category consisting of $W_n$-modules (resp. $(A_n, W_n)$-modules) $M $ on which $\frac{\partial}{\partial{t_i}}-a_i$ acts locally nilpotent for any $i\in\{1,\cdots,n\}$ and having finite dimensional Whittaker vector spaces. We show that each block $\Omega_{\mathbf{a}}^{\widetilde{W}}$  is equivalent to the category of finite dimensional modules over $L_n$, where $L_n=\text{Span}\{f(t)\frac{\partial}{\partial{t_i}}\mid f(t)\in A_n, f(0)=0, i=1,\dots,n\}$, see Theorem \ref{cat-equ}. When $\ba\in (\C^*)^n$, $\Omega_{\mathbf{a}}^W$ is said to be non-singular. We show that any module in  the non-singular block $\Omega_{\mathbf{a}}^W$ is a free
$U(\mh_n)$-module of finite rank when restricted to $U(\mh_n)$, where
$\mh_n$ is a Cartan subalgebra of $W_n$. Using the covering method introduced in \cite{BF1}, we show that any simple module in the non-singular block  $\Omega_{\mathbf{a}}^W$ is isomorphic to some simple sub-quotient of $T(A_n^{\ba}, V)$ for some finite dimensional simple  $\gl_n$-module $V$, see Theorem \ref{c-t}.   We also define the universal Whittaker module  $Q_{\ba}=U(W_n)\otimes_{U(\Delta_n)} \C_{\ba}$ and its endomorphism algebra $H_{\ba}=\text{End}_{W_n}(Q_{\ba})^{\text{op}}$. When $\ba\in (\C^*)^n$, we show that   $\Omega_{\mathbf{a}}^W$ and the category of finite dimensional $H_{\ba}$-modules are equivalent, see Theorem \ref{s-e}. A similar equivalence for finite $W$-algebras $W(\mg, e)$ was shown by Serge Skryabin, see the appendix in \cite{P}.

We denote by $\Z$, $\N$, $\Z_{\geq 0}$, $\C$ and $\C^*$ the sets of integers, positive integers, nonnegative
integers, complex numbers, and nonzero complex numbers, respectively.  For a Lie algebra
$L$ over $\C$, we use $U(L)$ to denote the universal enveloping algebra of $L$.

\section{Preliminaries}

\subsection{Witt algebras $W_n$}
Throughout this paper, $n$ denotes a positive integer.
We fix the vector space $\mathbb{C}^n$ of $n$-dimensional complex vectors.
Denote its standard basis by $\{\be_1,\be_2,\dots,\be_n\}$.

Denote by $A_n=\C[t_1,t_2,\cdots,t_n]$ the polynomial algebra in the commuting variables $t_1,t_2,\cdots,t_n$.
Let $\p_{i}=\frac{\p}{\p{t_i}}, h_i=t_i\partial_i$ for any $i=1, 2, \cdots, n$.
For the convenience,  set
$t^\bm=t_1^{m_1}\cdots t_n^{m_n}$, $h^{\bm}=h_1^{m_1}\cdots h_n^{m_n}$ and $|\bm|=m_1+\cdots+m_n$ for any $\bm=(m_1,\dots,m_n)\in\Z_{\geq 0}^n$.
Recall that  the classical Witt algebra $ W_n=\text{Der}(A_n)=\oplus_{i=1}^n{ A}_n\p_{i}$ is a free $A_n$-module of rank $n$. It  has the following Lie bracket:
$$[t^{\bm}\p_{i},t^{\br}\p_{j}]=
r_it^{\bm+\br-\be_i}\p_{j}
-m_jt^{\bm+\br-\be_j}\p_{i},$$
where all $\bm,\br\in \Z_{\geq 0}^n$.

It is well-known that  $W_n$ is a simple Lie algebra.
The subspace $\mathfrak{h}_n=\mathbb{C} h_1\oplus \cdots \oplus \mathbb{C} h_n $ is a Cartan subalgebra (a maximal
abelian subalgebra that is diagonalizable on $W_n$ with respect to the
adjoint action) of $W_n$.

Note that $A_n$ is a $W_n$-module under the derivation action of $W_n$. We denote the semidirect Lie algebra $W_n\ltimes A_n$ by $\widetilde{W}_n$ which is called the extended Witt algebra. Let $W_n$-$\text{Mod}$  be the category of finitely generated left $U(W_n)$-modules.

\begin{definition}A $W_n$-module $M$ is called a weight module if the action of $\mh_n$ on $M$
is diagonalizable, i.e.,
$M=\bigoplus_{\lambda \in \mh_n^{*}}M_{\lambda},$ where
$$M_{\lambda}=\{v \in M\mid hv=\lambda(h)v,  \forall \ h \in \mh_n\}.$$
\end{definition}

A weight $W_n$-module $M$ is  uniformly bounded  if the dimension of each weight space $M_{\lambda}$ is smaller than a fixed integer.
\begin{definition}A module $M$ over $\widetilde{W}_n$ is called an $(A_n,W_n)$-module if the action  of $A_n$ on $M$ is associative, i.e., $g(fv)=(gf)v$, for any $f,g\in A_n, v\in M$.
\end{definition}

\begin{example}\begin{enumerate}[$($a$)$]
\item The natural module $A_n$ is an $(A_n,W_n)$-module.
\item The adjoint $W_n$-module $W_n$ is an $(A_n,W_n)$-module under the following action: $$f(t)\cdot g(t)\p_i=(f(t)g(t))\p_i,\  \ f(t),g(t)\in A_n,$$
    where $i\in \{1,\dots,n\}$.
\end{enumerate}
\end{example}
\subsection{Whittaker modules}

Denote $\Delta_n=\span\{\p_1,\dots,\p_n\}$ which is an abelian subalgebra of $W_n$.
Since the adjoint action of $\Delta_n$ on $ W_n/\Delta_n$ is locally nilpotent. Thus $(W_n,\Delta_n)$ is a Whittaker pair in the sense of  \cite{BM}. A $W_n$-module $M$ is called a   Whittaker module if
the action of $\Delta_n$ on $M$ is locally finite.
For an $\mathbf{a}=(a_1,\cdots,a_n)\in \C^n$, we can define a
Lie algebra homomorphism $\phi_{\mathbf{a}}: \Delta_n\rightarrow \C$ such that $\phi_{\mathbf{a}}(\p_i)=a_i$ for any $i\in \{1,\cdots,n\}$. A   Whittaker module $M$ is of type $\phi_{\mathbf{a}}$ if for any $v\in M$ there is a $k\in \N$ such that $(x-\phi_{\mathbf{a}}(x))^kv=0$ for all $x\in \Delta_n$.
We also define the subspace of $M$ as follows:
 $$\text{Wh}_{\mathbf{a}}(M)=\{v\in M \mid x v=\phi_{\mathbf{a}}(x)v, \ \forall\ x\in \Delta_n \}$$ of $M$.
 An element in $\text{Wh}_{\mathbf{a}}(M)$ is called a Whittaker vector.

 We consider  Whittaker modules under some natural finite condition.  Let $\Omega_{\mathbf{a}}^W$ be the category consisting of   Whittaker $W_n$-modules of type $\phi_{\mathbf{a}}$ such that $\dim \text{Wh}_{\mathbf{a}}(M)< +\infty $. Similarly let $\Omega_{\mathbf{a}}^{\widetilde{W}}$ be the category consisting of Whittaker $(A_n, W_n)$-modules of type $\phi_{\mathbf{a}}$ such that $\dim \text{Wh}_{\mathbf{a}}(M)< +\infty $.

For a vector $\ba\in \C^n$, we call it non-singular if $a_i\neq 0$ for any $i\in \{1,\cdots,n\}$.

\subsection{Weighting functor}
In this subsection, we recall the weighting functor introduced in
\cite{N2}, which  establishes a relationship between weight modules and non-weight modules.
For an $\br=(r_1,\dots,r_n)\in\C^n$,  let $I_\br$ be the maximal ideal of  $U(\mh_n)$ generated by $$h_1-r_1,\dots, h_n-r_n.$$
For any $W_n$-module $M$ and  $\br\in\C^n$, set $M^{\br}:= M/I_{\br}M$.
Denote $$\mathfrak{W}(M):=\bigoplus_{\br\in\Z^n}M^{\br}.$$

By Proposition 8 in \cite{N2}, we have the following construction.
\begin{proposition} The vector space $\mathfrak{W}(M)$  becomes a weight $W_n$-module   under the following action:
\begin{equation}\label{3.3}
 t^{\bm+\be_i}\p_i\cdot(v+I_{\br}M):=  t^{\bm+\be_i}\p_i v+I_{\br+\bm}M, v\in M, \br\in \Z^n,
\end{equation} where $\bm\in \Z_{\geq 0}^n$ with $m_i\geq -1$.
\end{proposition}

In Section 4, we will use the  weighting functor to show that the $(A_n, W_n)$-cover $\widehat{M}$ has a finite composition series as an $(A_n, W_n)$-module for any module $M$ in  $\Omega_{\mathbf{a}}^W$ when $\ba$ is non-singular.

\section{Whittaker $(A_n,W_n)$-modules}

In this section, we study Whittaker $(A_n,W_n)$-modules. Explicitly, we establish an equivalence between $\Omega_{\mathbf{a}}^{\widetilde{W}}$ and
the category of finite dimensional $L_n$-modules.

\subsection{Whittaker modules over $\mathcal{D}_n$}
Let $\mathcal{D}_{n}$ be the  Weyl algebra
 over $A_n$ which is the unital  associative algebra
over $\mathbb{C}$ generated by $t_1,\dots,t_n$,
$\partial_1,\dots,\partial_n$ subject to
the relations
$$[\partial_i, \partial_j]=[t_i,t_j]=0,\qquad [\partial_i,t_j]=\delta_{i,j},\ 1\leq i,j\leq n.$$
A $\mathcal{D}_n$-module $M$ is called a Whittaker module of type $\ba$ if for any $v\in M$ there is a $k\in \N$ such that $(\p_i-a_i)^kv=0$ for all $i\in \{1,\cdots,n\}$.
Let $\mathcal{V}_{\mathbf{a}}^D$ be the category consisting of   finite length Whittaker $\mathcal{D}_n$-modules of type $\mathbf{a}$.

Let $\sigma_{\ba}$ be the algebra automorphism of $\mathcal{D}_n$ defined by $$t_i\mapsto t_i, \p_i \mapsto \p_i+a_i.$$
The $\mathcal{D}_n$-module $A_n$ can be twisted by $\sigma_{\ba}$ to be a new $\mathcal{D}_n$-module $A_n^{\ba}$.  Explicitly the the action of
$\mathcal{D}_n$ on  $A_n^{\ba}$ is defined by
$$X\cdot f(t)=\sigma_{\ba}(X)f(t), \ \ f(t)\in A_n, X\in \mathcal{D}_n.$$

\begin{lemma}\label{semi-simple}
\begin{enumerate}[$($a$)$]
\item Any simple module in $\mathcal{V}_{\mathbf{a}}^D$ is isomorphic to $A_n^{\ba}$.
\item The category $\mathcal{V}_{\mathbf{a}}^D$  is semi-simple.
\end{enumerate}
\end{lemma}
\begin{proof}(a) For any $f(t)\in A_n^{\ba}$, the formula  $(\p_i-a_i)\cdot t^{\bm}=m_it^{\bm-\be_i}$ tells us that
the degree of $t_i$ in $(\p_i-a_i)f(t)$ is smaller than that of $t_i$ in $f(t)$. By induction on the degree of $f(t)$, any submodule of $A_n^{\ba}$
must contain $1$ which generates $A_n^{\ba}$. So $A_n^{\ba}$ is a simple
$\mathcal{D}_n$-module. If $V$ is a simple module in $\mathcal{V}_{\mathbf{a}}^D$, then there is a nonzero $v\in V$ such that $\p_iv=a_iv$ for any $i$. Then the map $V\rightarrow A_n^{\ba}, t^{\bm}v\mapsto t^{\bm}$ defines a  $\mathcal{D}_n$-module isomorphism between
$V$ and $A_n^{\ba}$.

(b) Since $A^{\mathbf{a}}_{n}$ is equivalent to $A_n$ up to the algebra automorphism $\sigma_{\ba}$ of $\cd_n$ defined by $t_i\mapsto t_i, \p_i \mapsto \p_i+a_i$, $\text{Ext}^1_{\mathcal{D}_n}(A_n^{\ba}, A_n^{\ba})\simeq\text{Ext}^1_{\cd_n}(A_n, A_n)$.  By Lemma 3 in \cite{LZ2}, $\text{Ext}^1_{\cd_n}(A_n, A_n)=0$. So $\mathcal{V}_{\mathbf{a}}^D$  is semi-simple.
\end{proof}

\subsection{ The Whittaker category $\Omega_{\mathbf{a}}^{\widetilde{W}}$}

Since  $A_n$ is a left  module algebra over the Hopf algebra $U(W_n)$, we have the smash product algebra $A_n\# U(W_n)$. An $(A_n,W_n)$-module is actually a module
over $A_n\# U(W_n)$.

Let $L_n$ be the Lie subalgebra of $W_n$ spanned by $t^{\bm}\p_i$,  $i=1,\dots,n,  \bm\in\Z^n_{\geq 0}$ with $|\bm|\geq 1$ which is called the jet Lie algebra of $W_n$.

The following isomorphism was given   in \cite{XL1} or  \cite{BIN}.
\begin{theorem}\label{mainth}
The linear map $\phi:A_n\#U(W_n)\rightarrow \mathcal{D}_n\otimes U(L_n)$ defined by
$$\aligned\phi(t^{\bm}\partial_k)&=
t^{\bm}\partial_k \otimes 1+\sum_{\br\neq 0,\atop r_i\leq m_i}\binom{\bm}{\br} t^{\bm-\br} \otimes t^{\br}\partial_k,\\
\phi(t^{\bm})&=t^{\bm}\otimes 1, \ \ \  \bm\in\Z_{\geq 0}^n, k\in \{1,\cdots,n\},\endaligned$$  is an associative algebra isomorphism,
where $\binom{\bm}{\br}=\binom{m_1}{r_1}\cdots \binom{m_n}{r_n}$.
\end{theorem}
For any $\mathcal{D}_n$-module $P$, and $L_n$-module $V$, the tensor product $P\otimes V$ becomes an $(A_n,W_n)$-module $T(P,V)$ under the map
$\phi$.

The isomorphism in Theorem \ref{mainth} tells us that there is  a close relationship
between vector field Lie algebras and Weyl algebras.

Let $L_n\text{-mod}$ be the category of finite dimensional $L_n$-modules. When $V$ is finite dimensional,
we can check that $T(A_n^{\ba}, V)\in \Omega_{\mathbf{a}}^{\widetilde{W}}$ and the Whittaker vectors space $\text{Wh}_{\mathbf{a}}(T(A_n^{\ba}, V)) =V$.

\begin{lemma} \label{W-vector}For any $M\in \Omega_{\mathbf{a}}^{\widetilde{W}}$, the
Whittaker vector subspace $\text{Wh}_{\mathbf{a}}(M)$ is an $L_n$-module $V$ and $T(A_n^{\ba}, V)\cong M$.
\end{lemma}
\begin{proof} Through the isomorphism $\phi$ in Theorem \ref{mainth}, we can view $M$ as a module over $\cd_n\otimes U(L_n)$. In $\cd_n\otimes U(L_n)$, $[\cd_n, L_n]=0$. So $\text{Wh}_{\mathbf{a}}(M)$ is an $L_n$-module.
We claim that the $\cd_n\otimes U(L_n)$-module homomorphism
$$\aligned \xi: A_n^{\ba}\otimes \text{Wh}_{\mathbf{a}}(M) &\rightarrow M,\\
 t^{\bm}\otimes v  &\mapsto t^{\bm} v, \ \bm\in \mathbb{Z}_{\geq 0}^n,
\endaligned$$ is an isomorphism. By  Lemma \ref{semi-simple}, $\xi$ is surjective. Suppose that $w=\sum_{\bs\in \Lambda} t^{\bs}\otimes v_{\bs}\in \ker \xi$, where $\Lambda $ is a finite subset of $\mathbb{Z}_{\geq 0}^n$.
For  any $\bm\in \Lambda$,
by the simplicity of $A_n^{\ba}$ and the density theorem, there is some
$X\in \cd_n$ such that $X t^{\bm} =1$ and $X t^{\bs}=0$ for any $\bs\in \Lambda$ with $\bs\neq \bm$.  Then $X\xi(w)=\xi(Xw)=\xi(1\otimes v_{\bm})=v_{\bm}=0$.  Consequently $w=0$. So $\xi$ is also injective. Therefore, $\xi$ is an isomorphism.
\end{proof}

In the following theorem, we show that each block $\Omega_{\mathbf{a}}^{\widetilde{W}}$ is equivalent to
$L_n\text{-mod}$.
\begin{theorem}\label{cat-equ}
 The functor
$$\aligned\mathcal{F}: L_n\text{-mod}&\rightarrow \Omega_{\mathbf{a}}^{\widetilde{W}},\\
V&\mapsto T(A_n^{\ba}, V),
\endaligned$$
 is an equivalence of the two categories.
\end{theorem}
\begin{proof}  By Schur's Lemma, $\text{End}_{\cd_n}(A_n^{\ba})=\C$. From Lemma \ref{W-vector}, for any  $M\in \Omega_{\mathbf{a}}^{\widetilde{W}}$,
there is an $L_n$-module $\text{Wh}_{\ba}(M)$ such that
 $ \mathcal{F}(\text{Wh}(M)) \cong M$. From $\text{Hom}_{\cd_n\otimes U(L_n)}(A_n^{\ba}\otimes V, A_n^{\ba}\otimes W)\cong \text{Hom}_{U(L_n)}(V,W)$,
we see that
the homomorphism $$\mathcal{F}_{V,W}: \text{Hom}_{L_{n}}(V,W)\rightarrow
\text{Hom}_{\widetilde{W}_n}(\mathcal{F}(V),\mathcal{F}(W))
$$ is an isomorphism,
 for any modules $V,W\in L_n\text{-mod}$. Therefore $\mathcal{F}$ is an equivalence of the two categories.
\end{proof}

Next we give all simple Whittaker $(A_n,W_n)$-modules using the characterization of finite dimensional simple $L_n$-modules.  Let $\mathfrak{m}$ be the maximal ideal of  $A_n$ generated by
 $t_1,\cdots, t_n$. Then $\mm\Delta_n$ is a Lie subalgebra of $W_n$ which is equal to $L_n$.

The following lemma was given in \cite{XL1}.
 \begin{lemma}\label{xl-lemma}
 \begin{enumerate}[$($a$)$]
\item The linear map $\pi: \mm\Delta_n/\mm^2 \Delta_n\rightarrow \gl_n$ such that
 $$\aligned
t_i\p_j&\mapsto e_{ij},\ \ i,j\in\{1,\cdots,n\},
 \endaligned$$
 is a Lie algebra isomorphism.
\item If $V$ is a finite dimensional simple $L_n$-module, then $ \mm^2 \Delta_n V=0$. So $V $ is a simple $\gl_n$-module
via the isomorphism in (a).
\end{enumerate}
\end{lemma}
By Lemma \ref{xl-lemma} and Theorem \ref{mainth}, for any simple finite dimensional $\gl_n$-module $V$,  we can define an $(A_n,W_n)$-module structure on $T(A_n^{\ba}, V)$ as follows:
$$\aligned t^{\bm}\partial_k(p\otimes v)& =
t^{\bm}\partial_k p\otimes v+\sum_{i=1}^n m_it^{\bm-\be_i} p\otimes e_{ik}v,\\
t^{\bm}(p\otimes v) &=(t^{\bm}p)\otimes v, \ \ \  \bm\in\Z_{\geq 0}^n, k\in \{1,\cdots,n\},\endaligned$$ where $p\in  A_n^{\ba}, v\in V$.

Combining Theorem \ref{cat-equ} with Lemma \ref{xl-lemma}, we obtain the following description of simple Whittaker $(A_n,W_n)$-modules.

\begin{corollary} \label{AW-mod}If $M$ is a simple Whittaker $(A_n,W_n)$-module of type $\phi_{\ba}$ such that $\dim \text{Wh}_{\ba}(M)<\infty$, then $M\cong T(A_n^{\ba}, V)$ for some finite dimensional simple $\gl_n$-module $V$.
\end{corollary}

\section{Whittaker $W_n$-modules}

Throughout this section, we always assume that $\ba$ is non-singular, i.e., $\ba\in (\C^*)^n$. In this case,  a module in $\Omega_{\mathbf{a}}^W$ is said to be non-singular. In this section,  we classify simple $W_n$-modules in $\Omega_{\mathbf{a}}^W$.  We also obtain an analogue of Skryabin's equivalence for  the category $\Omega_{\mathbf{a}}^W$.

\subsection{Preliminary lemmas}

The following result gives an interesting property of non-singular Whittaker modules.

\begin{lemma}\label{finite} If $M$ is a  Whittaker $W_n$-module of type $\phi_{\ba}$, then $$M=U(\mh_n)\text{Wh}_{\ba}(M):=\text{Span} \{ h^{\bm} v\mid v\in  \text{Wh}_{\ba}(M), \bm \in \Z_{\geq 0}^n\}.$$ Moreover, if $\{ v_i\mid i\in \Lambda\}$ is a basis of $\text{Wh}_{\ba}(M)$,  then $$\{ h^{\bm}v_i \mid \bm\in \Z_{\geq 0}^n, i\in \Lambda \}$$ is a basis of $M$.
\end{lemma}
\begin{proof}
For an $\bm\in \Z_{\geq 0}^n$, define
$$ \p_{\bm}=(\p_1-a_1)^{m_1}\cdots(\p_n-a_n)^{m_n}.$$
Note that the elements $\p_{\bm}$ with $ \bm\in\Z_{\geq 0}^n$ form a basis for $U(\Delta_n)$. Consider the total order on $\Z_{\geq 0}^n$ satisfying the condition: $\br<\bm$ if $|\br|<|\bm|$ or
$|\br|=|\bm|$ and
there is an $l\in \{1,\cdots, n\}$ such that $r_i=m_i$ when $1\leq i <l$ and $r_l< m_l$. For each $\bm$, the set $\{\br \in \Z_{\geq 0}^n\mid \br< \bm\}$ is a finite set. So $\Z_{\geq 0}^n$ is isomorphic to $\Z_{\geq 0}$ as an ordered set.
For a nonzero $\bm \in \Z_{\geq 0}^n$, denote by $\bm'$ the predecessor
of $\bm$.

Using the condition that $a_i\neq 0$ for any $i$ and induction on $\bm$, we can show that for any $\bm,\bs\in \Z_{\geq 0}^n$  and nonzero $v\in \text{Wh}_{\ba}(M)$, we have that $\p_{\bs} h^{\bm}v=0$ whenever $\bs>\bm$,
$\p_{\bm} h^{\bm} v= k_{\bm} v$ for some nonzero scalar $k_{\bm}$.

For each $\bm \in \Z_{\geq 0}^n$,  denote by $\mathbb{I}_{\bm}$  the ideal of $U(\Delta_n)$ spanned by  $\p_{\bs}$ with $\bs>\bm$, and $M_{\bm}=\{w\in M\mid \mathbb{I}_{\bm}w=0\}$. Clearly $\text{Wh}_{\ba}(M)=M_{\bo}$.
For any nonzero $w\in M$, by the definition of Whittaker modules, there is an
$\bm \in \Z_{\geq 0}^n$ such that $w\in M_{\bm}\setminus M_{\bm'}$, i.e.,
$\p_{\bm} w\neq 0$ and $\p_{\bs} w=0$ for any $\bs>\bm$. So $\p_{\bm} w\in \text{Wh}_{\ba}(M)$. By the above discussion, $\p_{\bm} h^{\bm} \p_{\bm} w=k_{\bm}\p_{\bm} w  $.
We call $\bm$ the degree of $w$. We use  induction on the degree  $\bm$ of $w$ to show that $w\in U(\mh_n)\text{Wh}_{\ba}(M)$. Let $w'=w-\frac{1}{k_{\bm}}h^{\bm}\p_{\bm}w$. Then
$$\p_{\bm} w'=\p_{\bm} w-\frac{1}{k_{\bm}}\p_{\bm} h^{\bm}\p_{\bm}w=0. $$
This implies that the degree of $w'$ is smaller than $\bm$. By the induction
hypothesis, $w' \in U(\mh_n)\text{Wh}_{\ba}(M)$. Hence $w\in U(\mh_n)\text{Wh}_{\ba}(M)$.

Suppose that $w:=\sum_{\mathbf{r} \leq \mathbf{m}} \sum_{i=1}^k c_{\mathbf{r},i}h^{\mathbf{r}} v_i=0$ in $M$, where $c_{\mathbf{r},i}\in \C$.
From
$\p_{\bm} w=0$, we see that $c_{\bm,i}=0$. Then  by induction on $\bm$, we have that
$c_{\br,i}=0$ for any $\br < \bm$ and $i$. Thus  $\{ h^{\mathbf{m}}v_i\mid \mathbf{m}\in \Z_{\geq 0}^n, i\in \Lambda\}$ is linearly independent. In conclusion, we complete the proof.

\end{proof}

From Lemma \ref{finite},  if $M$ is a  Whittaker $W_n$-module of type $\phi_{\ba}$, then
$M$ is a free $U(\mh_n)$-module and $\text{Wh}_{\ba}(M)$ is a basis of
$M$ as  a free $U(\mh_n)$-module.

\begin{corollary}\label{torsion-1}
Any nontrivial  $W_n$-module $M\in\Omega_{\mathbf{a}}^W$ is a  free $U(\mh_n)$-module of finite rank.
\end{corollary}

\begin{corollary}\label{torsion-2}
The category $\Omega_{\mathbf{a}}^W$ is an abelian category.
\end{corollary}
\begin{proof}Suppose that $M\in \Omega_{\mathbf{a}}^W$ and $N$ is an arbitrary $W_n$-submodule of $M$. Since $\text{Wh}_{\ba}(N)\subset \text{Wh}_{\ba}(M)$, $N\in \Omega_{\mathbf{a}}^W$. By Lemma \ref{finite}, $M/N$ is a free $U(\mh_n)$ module whose rank is not bigger than that of $M$.
 Thus $\dim \text{Wh}_{\ba}(M/N)< \infty$ and hence $M/N\in \Omega_{\mathbf{a}}^W$.
\end{proof}

Next let us recall the $(A_n, W_n)$-cover introduced in \cite{BF1} in a slightly
different form.
Let $\Delta$ be the co-multiplication of $U(\widetilde{W}_n)$, that is
$$\Delta: U(\widetilde{W}_n)\rightarrow U(\widetilde{W}_n)\otimes U(\widetilde{W}_n), X\mapsto X\otimes 1+1\otimes X,\  \ X\in \widetilde{W}_n.$$
Let $\pi: U(\widetilde{W}_n)\rightarrow U(W_n)$ be the homomorphism such that
$\pi |_{W_n}=\text{id}_{W_n}, \pi |_{A_n}=0$. Denote the composition $$(\text{id}_{\widetilde{W}_n}\otimes \pi)\cdot\Delta: U(\widetilde{W}_n)\rightarrow U(\widetilde{W}_n)\otimes U(W_n)$$ by $\Delta'$. Then for any $(A_n,W_n)$-module $P$ and $W_n$-module $M$, $P\otimes M$ becomes an $(A_n,W_n)$-module under the pulling back of $\Delta'$. In particular,
$W_n\otimes M$ can be defined to be an $(A_n,W_n)$-module by
$$t^{\br}\cdot (t^{\bm}\p_i\otimes w)=t^{\bm+\br}\p_i\otimes w,\ \ w\in M, \br,\bm\in \Z_+^n, i\in\{1,\cdots,n\}.$$

Define a linear map $\theta: W_n\otimes M\rightarrow M$ such that
$$t^{\bm}\p_i\otimes w\mapsto t^{\bm}\p_i w, \ w\in M,\bm\in \Z_+^n, i\in\{1,\cdots,n\}.$$
One can check directly that $\theta$ is a $W_n$-module homomorphism.
Define
$$K(M)=\{ v\in \ker\theta \mid A_n\cdot v\subset \ker \theta\}$$
which is an $(A_n,W_n)$-submodule of $W_n\otimes M$.
Let $\widehat{M}=(W_n\otimes M)/K(M)$ which is called the $(A_n, W_n)$-cover of
$M$. It is natural that $\theta$ induces a $W_n$-module homomorphism
$\hat{\theta}: \widehat{M}\rightarrow M$. For any $X\in W_n$, $v\in M$,
denote the image of $X\otimes v$ in $\widehat{M}$ by $X\boxtimes v$.

\subsection{Simple Whittaker modules}

For any $m\in \Z_{\geq 0}$, $j\in \{1,\dots,n\}$ and $\bal,\bbe\in \Z^n_{\geq 0}$, in $U(W_n)$, we denote the following operator:
$$\omega_{\bal,\bbe}^{m,j,l,p}:=\sum_{i=0}^m (-1)^i \binom{m}{i}t^{\bal+(m-i)\be_j}\p_l\cdot t^{\bbe+i\be_j}\p_p. $$

The following lemma was given by Xue and Lu, see \cite{XL1}.
\begin{lemma}\label{zero}If $M$ is a uniformly bounded weight module over $W_n$, then
there is an $m\in \N$ such that  $\omega_{\bal,\bbe}^{m,j,l,p}M=0$ for all
$j,l,p\in \{1,\dots,n\}$ and $\bal,\bbe\in \Z^n_{\geq 0}$.
\end{lemma}

Using Lemma \ref{zero}, we show that for any simple  module $M$ in
$\Omega_{\mathbf{a}}^W$,  $\text{Wh}_{\ba}(\widehat{M})$ is finite dimensional.
\begin{lemma}\label{fin}Suppose that $M$ is a simple module in $\Omega_{\mathbf{a}}^W$.
Then $\text{Wh}_{\ba}(\widehat{M})$ is finite dimensional.
\end{lemma}

\begin{proof} By Lemma \ref{finite} and Corollary \ref{torsion-1},   $M$  is a   free  $U(\mh_n)$-module of finite rank.  Thus  $\mathfrak{W}(M)$ is a uniformly bounded weight $W_n$-module. In fact each weight space of
$\mathfrak{W}(M)$  has dimension no bigger than the rank of $M$
as a free $U(\mh_n)$-module.  From Lemma \ref{zero}, there exists $m\in \N$ such that
$$\omega_{\br,\bs}^{m,j,l,p}\mathfrak{W}(M)=0,$$
for all  all
$j,l,p\in \{1,\dots,n\}$ and $\br,\bs\in \Z^n_{\geq 0}$.

By the definition of $\mathfrak{W}(M)$, we deduce that
$$\omega_{\br,\bs}^{m,j,l,p}M\subseteq \bigcap_{\bal\in \Z^n} I_{\bal}M=0,\ \forall\  j,l,p\in \{1,\dots,n\}, \br,\bs\in \Z^n_{\geq 0}, $$ where $I_\bal$ is the maximal ideal of  $U(\mh_n)$ generated by $h_1-\alpha_1,\dots, h_n-\alpha_n$. The result $\bigcap_{\bal\in \Z^n} I_{\bal}M=0$ follows from that $M$  is a  finitely generated free  $U(\mh_n)$-module.

The simplicity of $M$ implies that
 $M=W_n M$. Next  we prove by induction on $|\br|$ that
$$t^{\br}\p_l\boxtimes t^{\bs}\p_p v\in \sum_{i=1}^n\sum_{|\bal|\le mn}  t^{\bal} \p_i \boxtimes M,$$
for all $v\in M, \br,\bs\in\Z^n, p,l\in\{1,\cdots,n\}.$
This is clear  for $\br\in\Z^n_{\geq 0}$ with $|\br|\le mn$. Now we assume that $|\br|>mn$. Without lose of generality, we may assume that $r_1>m$.

For any $\br\in\Z^n_{\geq 0}$, we have
$$\aligned
 &\theta(\sum_{i=0}^m (-1)^i {m\choose i} t^{\br-i\be_1}\p_l\otimes ( t^{\bs+i\be_1}\p_pv ))\\
=&\sum_{i=0}^m (-1)^i {m\choose i}t^{\br-i\be_1}\p_l  t^{\bs+i\be_1}\p_pv\\
=&\omega_{\br,\bs}^{m,1,l,p}(v)=0,
\endaligned$$
which implies
$$\sum_{i=0}^m (-1)^i {m\choose i} t^{\br-i\be_1}\p_l\otimes t^{\bs+i\be_1}\p_pv \in K(M).$$
This means that
$$ \aligned t^{\br}\p_l\boxtimes t^{\bs}\p_pv =\sum_{i=1}^m (-1)^i {m\choose i} t^{\br-i\be_1}\p_l\boxtimes t^{\bs+i\be_1}\p_pv ,
\endaligned$$ in
$\widehat{M}$
which belongs to $ \sum_{i=1}^n\sum_{|\br|\le mn}  t^{\br} \p_i \boxtimes M$
by the induction hypothesis.

Therefore $\widehat{M}$ is a finitely generated $U(\mh_n)$-module. By Lemma \ref{finite},  $\text{Wh}_{\ba}(\widehat{M})$ is finite dimensional. Otherwise, $\widehat{M}$ is a free $U(\mh_n)$-module of infinite rank.
\end{proof}

\begin{theorem}\label{c-t} If $M$ is a simple module in $\Omega_{\mathbf{a}}^W$, then $M$ is isomorphic to some simple sub-quotient of $T(A_n^{\ba}, V)$ for some finite dimensional simple $\gl_n$-module $V$.
\end{theorem}

\begin{proof} From Lemma \ref{fin},  there is an $(A_n,W_n)$-module  $M_1$  such that  $\text{Wh}_{\ba}(M_1)$ is finite dimensional and
a $W_n$-module epimorphism $\psi: M_1\rightarrow M$. We may choose $M_1$ such that $\dim \text{Wh}_{\ba}(M_1)$  is minimal.
The minimality of $M_1$ forces that $M_1$ is a simple  $(A_n,W_n)$-module. Otherwise, $M_1$ has a nonzero maximal $(A_n,W_n)$-submodule $M_2$.
Then from corollary \ref{torsion-1}, $M_2$ and $M_1/M_2$ are free $U(\mh_n)$  module of rank both less than $\dim \text{Wh}_{\ba}(M_1)$. Because $M$ is a simple $W_n$-module, we must have either $\psi(M_2)=M$ or $\psi(M_2)=0$. Then either $M_2$ or $M_1/M_2$ has a simple $W_n$-quotient module isomorphic to $M$,  contradicting  the minimality of $M_1$. From   Corollary \ref{AW-mod}, we know that $M_1\cong T(A_n^{\ba}, V)$ for a finite dimensional simple $\gl_n$-module $V$. Hence $M$ is a simple $W_n$-quotient of  $T(A_n^{\ba}, V)$.\end{proof}

The structure of $T(A_n^{\ba}, V)$ was determined in \cite{LLZ}. Let $\C^n$ be the natural $n$-dimensional representation of $\gl_n$ and let $V(\delta_k,k)$ be its
$k$-th exterior power, $k = 1,\cdots,n$. We use
 $V(\delta_0,
0)$ to denote  the $1$-dimensional trivial $\gl_n$-module. A module $V$ over a Lie algebra $\mg$ is called a trivial module if $X v=0$ for any $X\in \mg, v\in V$.
 For a finite dimensional simple $\gl_n$-module $V$, $T(A_n^{\ba}, V)$ is a simple module over
$W_n$ if  $V \not\cong V(\delta_k, k)$ for any $k\in \{1,\cdots, n\}$. Moreover
we have the following complex of $W_n$-modules:
\begin{equation}\label{complex}\xymatrix
{0\ar[r] &T(A_n^{\ba},V(\delta_0,0))\ar[r]^{\pi_0} &T(A_n^{\ba},V(\delta_1,1))\ar[r]^{\hskip 1cm\pi_1} &\cdots \\
\cdots\ar[r]^{\pi_{n-2}\hskip 2cm}&{T(A_n^{\ba},V(\delta_{n-1},n-1))}\ar[r]^{\hskip .8cm\pi_{n-1}} &T(A_n^{\ba},V(\delta_n,n))\ar[r]& 0,}
\end{equation}
where
$\pi_{k-1}$ is the $W_n$-module homomorphism defined by
\begin{equation*}\begin{array}{lrcl}
\pi_{k-1}:& T(A_n^{\ba},V(\delta_{k-1},k-1)) & \rightarrow & T(A_n^{\ba}, V(\delta_{k},k)),\\
      & p\otimes v & \mapsto & \sum_{j=1}^{n} \p_j\cdot p\otimes (\be_j\wedge v),
\end{array}\end{equation*}
for all $p\in A_n^{\ba}$ and $v\in V(\delta_{k-1}, k-1)$, where $T(A_n^{\ba}, V(\delta_{-1},-1)) =0$, $k\in \{0, 1,\cdots, n\}$.

\begin{lemma}\label{exact}\begin{enumerate}[$($a$)$]
\item  The complex (\ref{complex}) is exact.
\item  The $W_n$-modules $\text{Im}\pi_0, \cdots, \text{Im}\pi_{n-1}$ are simple.
\end{enumerate}
\end{lemma}
\begin{proof}
Then from Theorem 3.5 in \cite{LLZ}, the $W_n$-modules $\text{Im}\pi_{k-1}$ are simple, for $k=1,2,\ldots,n$. And $\text{Ker} \pi_{k}/\text{Im}\pi_{k-1}$, $T(A_n^{\ba}, V(\delta_n,n))/\text{Im}\pi_{n-1}$ are trivial $W_n$-modules, for $k=1,2,\ldots,n-1$. Since $a_i\neq 0$ for any $i$,  a trivial module in
  $\Omega_{\mathbf{a}}^W$ must be a zero module. So $\text{Ker} \pi_{k}=\text{Im}\pi_{k-1}$. Hence the complex (\ref{complex}) is exact.
\end{proof}
\subsection{An analogue of Skryabin's equivalence}

Let $\C_{\ba}=\C 1_{\ba}$ be the one dimensional $U(\Delta_n)$-module such that $\p_i 1_{\ba}= a_i 1_{\ba}$ for any $i\in\{1,\dots,n\}$.

Let $$Q_{\ba}=U(W_n)\otimes_{U(\Delta_n)} \C_{\ba},$$
an induced $W_n$-module, and
$$H_{\ba}=\text{End}_{W_n}(Q_{\ba})^{\text{op}},$$ an associative algebra over $\C$. Then $Q_{\ba}$ is both a left $U(W_n)$-module and a right  $H_{\ba}$-module. Let $H_{\ba}\text{-mod}$ be the category of finite dimensional $H_{\ba}$-modules.

\begin{theorem}\label{s-e}\begin{enumerate}[$($a$)$]
\item  As a right  $H_{\ba}$-module, $Q_{\ba}$ is free. More precisely, the subset
$$\{h^\bm \otimes 1_{\ba} \mid \bm\in \Z_{\geq 0}^n\}$$ is
a basis of  $Q_{\ba}$ as an  $H_{\ba}$-module.
\item The functors $M\mapsto \text{Wh}_{\ba}(M)$ and $V\mapsto Q_{\ba}\otimes_{H_{\ba}} V$ are inverse equivalence between  $\Omega_{\mathbf{a}}^W$ and $H_{\ba}\text{-mod}$ .
\end{enumerate}
\end{theorem}
\begin{proof}
(a).  By adjunction and the definition of Whittaker Vectors, we have $$\text{End}_{W_n}(Q_{\ba})\cong \text{Hom}_{\Delta_n}(\C_{\ba}, Q_{\ba})\cong \text{Wh}_{\ba}(Q_{\ba}).$$
Then $\text{Wh}_{\ba}(Q_{\ba})$ is a free $H_{\ba}$-module of rank $1$.

By Lemma \ref{finite}, $Q_{\ba}=U(\mh_n) \text{Wh}_{\ba}(Q_{\ba})$.
By PBW Theorem, $Q_{\ba}$ is $U(\mh_n)$-free. Then (a) follows.

(b).
 By (a) and Lemma \ref{finite}, for any $M\in \Omega_{\mathbf{a}}^W$,
 $$ Q_{\ba}\otimes_{H_{\ba}} \text{Wh}_{\ba}(M)\cong M.$$
  On the other side, for any $V\in H_{\ba}\text{-mod}$,  $$  \text{Wh}_{\ba}(Q_{\ba}\otimes_{H_{\ba}} V )=\text{Wh}_{\ba}(U(\mh_n)\otimes V )\cong V.$$
 Then we can complete the proof.
\end{proof}

 For finite $W$-algebras $W(\mg, e)$, a similar category  equivalence was shown by
  Serge Skryabin, see the appendix in \cite{P}.

\begin{remark} By Theorem \ref{c-t} and Theorem \ref{s-e}, we can obtain all finite dimensional simple $H_{\ba}$-modules.  It will be interesting  to characterize the category $H_{\ba}\text{-mod}$  of finite dimensional $H_{\ba}$-modules and find the  PBW basis of $H_{\ba}$. For finite $W$-algebras $W(\mg, e)$,
 the  PBW basis of $W(\mg, e)$ was given in \cite{P}. By Proposition 35 in \cite{BM}, $H_{\ba}$ is isomorphic to a subalgebra of  $U(L_n)$.
\end{remark}

\begin{center}
\bf Acknowledgments
\end{center}

\noindent   G.L. is partially supported by NSF of China(Grants
11771122).


\vspace{4mm}

\noindent   \noindent Y.Z.: School of Mathematics and Statistics,
Henan University, Kaifeng 475004, China. Email:  15518585081@163.com

\vspace{0.2cm}

 \noindent G.L.: School of Mathematics and Statistics,
and  Institute of Contemporary Mathematics,
Henan University, Kaifeng 475004, China. Email: liugenqiang@henu.edu.cn


\begin{thebibliography}{abcdsfgh}

\bibitem{ALZ} D. Adamovi\'c, R. Lu, K. Zhao, Whittaker modules for the affine Lie algebra $A^{(1)}_1$, Adv. Math. 289(2016), 438-479.

\bibitem{AP} D. Arnal, G. Pinczon,  On algebraically irreducible representations of the Lie
algebra $\sl_2$, J. Math. Phys. 15(1974), 350-359.


\bibitem{B} Y. Billig, Jet modules, Canad. J. Math. 59(2007), 712-729.



\bibitem{BF1} Y. Billig, V. Futorny, Classification of irreducible representations of Lie algebra of vector fields on a torus, J. Reine Angew. Math. 720 (2016), 199-216.


\bibitem{BIN} Y. Billig, C. Ingalls, A. Nasr, $\mathcal{A}\mathcal{V}$-modules of finite type on affine space, 	arXiv:2002.08388.

\bibitem{BM}P. Batra, V. Mazorchuk,
 Blocks and modules for Whittaker pairs, J. Pure Appl. Algebra 215 (7) (2011), 1552-1568.




\bibitem{BO} G. Benkart, M. Ondrus, Whittaker modules for generalized Weyl algebras, Represent. Theory 13(2009), 141-164.


\bibitem{C} C. Chen, Whittaker modules for classical Lie superalgebras, Comm. Math. Phys. 388 (2021), no. 1, 351-383.

\bibitem{CG} A. Cavaness, D. Grantcharov, Bounded weight modules of the Lie algebra of vector fields on $\C^2$,  J. Algebra Appl.16 (2017), no. 12, 1750236, 27 pp.

    \bibitem{CJ}    X. Chen, C.  Jiang,  Whittaker modules for the twisted affine Nappi-Witten Lie algebra $\widehat{H}_4[\tau]$,  J. Algebra 546 (2020), 37-61.

\bibitem{DSY}F. Duan, B. Shu, Y. Yao, The Category $\mathcal{O}$ for
Lie algebras of vector fields (I): Tilting modules and character formulas, Publ. Res. Inst. Math. Sci.  56 (2020), no. 4, 743-760.

\bibitem{Ch}  K. Christodoulopoulou, Whittaker modules for Heisenberg algebras and imaginary Whittaker modules for affine Lie algebras, J. Algebra 320(2002), 2871-2890.

\bibitem{E1} S. Eswara Rao, Irreducible representations of
the Lie algebra of the diffeomorphisms of a $d$-dimensional torus,
J. Algebra 182 (1996),  no. 2, 401-421.

\bibitem{E2} S. Eswara Rao, Partial classification of modules for Lie algebra of
diffeomorphisms of d-dimensional torus, J. Math. Phys., 45 (8),
(2004) 3322-3333.

 \bibitem{GLZ}  X. Guo, R. Lu, K. Zhao, Irreducible modules
 over the Virasoro algebra, Doc. Math. 16(2011), 709-721.


 \bibitem{GLLZ}X. Guo, G. Liu, R. Lu, K. Zhao, Simple Witt modules that are finitely generated over the Cartan subalgebra, Mosc. Math. J. 20(2020), no.1, 43-65.


 \bibitem{GS} D. Grantcharov, V. Serganova, Simple weight modules with finite weight multiplicities over the
Lie algebra of polynomial vector fields, arXiv: 2102.09064.

%

\bibitem{K}  B. Kostant, On Whittaker vectors and representation theory, Invent. Math. 48(1978), 101-184.

\bibitem{La} T. A. Larsson, Conformal fields: A class of representations of $\text{Vect}(N)$, Int. J. Mod. Phys. A 7
(1992), 6493-6508.

\bibitem{LLZ}G. Liu, R. Lu, K. Zhao, Irreducible Witt modules from Weyl modules and
$\gl_n$-modules, J. Algebra 511(2018), 164-181.


\bibitem{LZ2}G. Liu, K. Zhao, The category of weight modules for symplectic oscillator Lie algebras, (2019), Transformation Groups: DOI: 10.1007/s00031-021-09639-y.

\bibitem{LPX} D. Liu, Y. Pei, L. Xia, Whittaker modules for the super-Virasoro algebras. J Algebra Appl. 18 (2019), 1950211 (13 pp).

 \bibitem{LWZ}   D. Liu, Y. Wu, L. Zhu, Whittaker modules for the twisted Heisenberg-Virasoro algebra, J. Math. Phys. 51 (2010),  023524.

\bibitem{M1} O.~Mathieu, Classification of Harish-Chandra modules over the Virasoro Lie algebra,
Invent. Math.  107 (1992), no. 2, 225-234.


\bibitem{N2}  J. Nilsson, $U(\mh)$-free modules and coherent families, J. Pure Appl. Algebra 220 (4) (2016), 1475-1488.

\bibitem{O1} M. Ondrus, Whittaker modules for $U_q(\sl_2)$,  J. Algebra (2005), 289: 192-213.

\bibitem{OW1}     M. Ondrus M, E. Wiesner, Whittaker modules for the Virasoro algebra. J. Algebra Appl.  (2009), 8: 363-377.

\bibitem{OW2} M. Ondrus M, E. Wiesner, Whittaker categories for the Virasoro algebra, Comm Algebra, 2013, 41: 3910-3930.

\bibitem{PS}
{ I.~Penkov, V.~Serganova}, Weight representations of the
polynomial Cartan type Lie algebras $W\sb n$ and $\overline S\sb n$,
Math. Res. Lett., 6 (1999),  no. 3-4, 397-416.

\bibitem{P}  A.  Premet, Special transverse slices and their enveloping algebras, with an appendix by Serge Skryabin, Adv. Math. 170, no. 1 (2002), 1-55.

\bibitem{R1} A. N. Rudakov, Irreducible representations of infinite-dimensional Lie algebras of Cartan
type, Izv. Akad. Nauk SSSR Ser. Mat., 38 (1974), 836-866 (Russian); English translation in
Math USSR-Izv., 8 (1974), 836-866.

\bibitem{R2}A. N. Rudakov, Irreducible representations of infinite-dimensional Lie algebras of types S and H,
Izv. Akad. Nauk SSSR Ser. Mat., 39 (1975), 496-511 (Russian); English translation in Math
USSR-Izv., 9 (1975), 465-480.



\bibitem{Sh} G. Shen, Graded modules of graded Lie algebras of
Cartan type, I. Mixed products of modules,  Sci. Sinica Ser., A
29  (1986),  no. 6, 570-581.

\bibitem{S}   A. Sevostyanov, Quantum deformation of Whittaker modules and the Toda lattice,  Duke Math. J. 105(2000), 211-238.


\bibitem{XGZ} L. Xia, X. Guo, J. Zhang,
Classification on irreducible Whittaker modules over quantum group $U_q(\sl_3, \Lambda)$,
Front. Math. China 16 (2021), no. 4, 1089-1097.

\bibitem{XL1}    Y. Xue, R. Lu, Classification of simple bounded weight modules of the Lie algebra of vector fields
on $\C^n$, arxiv:2001.04204v1.





\end{thebibliography}
\end{document}